\documentclass[12pt,leqno]{amsart}
\usepackage{amsmath,amssymb,amsfonts}
\usepackage{eucal,enumerate}

\newtheorem{lem}{Lemma}[section]
\newtheorem{cor}[lem]{Corollary}
\newtheorem{prop}[lem]{Proposition}
\newtheorem{thm}[lem]{Theorem}

\theoremstyle{definition}
\newtheorem{qn}[lem]{Question}

\numberwithin{equation}{section}

\renewcommand{\phi}{\varphi}                 % Personal preferences.
\renewcommand{\epsilon}{\varepsilon}
\newcommand\eset{\varnothing}

\newcommand\setm{\setminus}
\newcommand\inv{^{-1}}
\newcommand\conv{\operatorname{conv}}
\newcommand\cC{\mathcal{C}}
\newcommand\bbF{\mathbb{F}}
\newcommand\bbR{\mathbb{R}}
\newcommand\bbZ{\mathbb{Z}}
\newcommand\G{\Gamma}
\renewcommand\S{\Sigma}
\newcommand\s{\sigma}

\newcommand\A{\emph{Ans. }}
\newcommand\Q{\emph{Qn.\ }}
\newcommand\Qs{\emph{Qns.\ }}

\begin{document}

\title{Negative (and Positive) Circles in Signed Graphs: \\ A Problem Collection}

\author{Thomas Zaslavsky}
\address{Binghamton University (SUNY) \\ Binghamton, N.Y., U.S.A.}
\email{\tt zaslav@math.binghamton.edu}

\begin{abstract}
A signed graph is a graph whose edges are labelled positive or negative.  The sign of a circle (cycle, circuit) is the product of the signs of its edges.  Most of the essential properties of a signed graph depend on the signs of its circles.  
Here I describe several questions regarding negative circles and their cousins the positive circles.  Topics include incidence between signed circles and edges or vertices, characterizing signed graphs with special circle properties, counting negative circles, signed-circle packing and covering, signed circles and eigenvalues, and directed cycles in signed digraphs.  
A few of the questions come with answers.
\end{abstract}



\date{\today}

\maketitle

\begin{center}
In honor and in memory of\\ {\sc Dr.~B.~Devadas~Acharya} (1947--2013)
\end{center}

%%%%%%%%%%%%%%%%%%%%%%%%%%%%%%%%%%%%%%%%%%%%%%%%%%%%%%%%%
%%%%%%%%%%%%%%%%%%%%%%%%%%%%%%%%%%%%%%%%%%%%%%%%%%%%%%%%%

%%%%%%%%%%%%%%%%%%%%%%%%%%%%%
%%%%%%%%%%%%%%%%%%%%%%%%%%%%%
\section*{Introduction}\label{intro}

A signed graph is a graph with a \emph{signature} that assigns to each edge a positive or negative sign.  To me the most important thing about a signed graph is the signs of its circles,\footnote{A circle is a connected, 2-regular graph.  The common name ``cycle'' has too many other meanings.}
 which are calculated by multiplying the signs of the edges in the circle.  Thus a signature is essentially its list of negative circles, or (of course) its list of positive circles.  I will describe some of the uses of and questions about circles of different signs in a signed graph.  Both theorems and algorithms will be significant.
  
The topic of this report is broad.  Of necessity, I will be very selective and arbitrarily so, omitting many fine contributions.  (Let no one take offense!)

I chose this topic in part because it has many fine open problems, but especially in honor of our dear friend Dr.~B.~Devadas Acharya---``our'' because he was the dear friend of so many.  Among Dr.~Acharya's wide combinatorial interests, I believe signed graphs were close to his heart, one of his---and his collaborator and wife's, Prof.~Mukti Acharya's---first and lasting areas of research.  Circles (or ``cycles'') in signed graphs exemplify well Dr.~B.~D.~Acharya's approach to mathematics, that new ideas and new problems are its lifeblood.  He himself was an enthusiastic and inspiring font of new ideas.  I hope some of his spirit will be found in this survey.

%%%%%%%%%%%%%%%%%%%%%%%%%%%%%
%%%%%%%%%%%%%%%%%%%%%%%%%%%%%
%
\section{Groundwork}\label{ground}

%%%%%%%%
\subsection{Signed graphs}\

A \emph{signed graph} $\S = (\G,\s) = (V,E,\s)$ is defined as an \emph{underlying graph} $\G=(V,E)$, also written $|\S|$, and a signature $\s: E \to \{+,-\}$ (or $\{+1,-1\}$), the sign group.  The sets of positive and negative edges are $E^+(\S)$ and $E^-(\S)$.  
In the literature $\G$ may be assumed to be simple, or it may not (this is graph theory); I do not assume simplicity.  Each circle and indeed each walk $W = e_1e_2\cdots e_l$ has a sign $\s(W) := \s(e_1)\s(e_2)\cdots\s(e_l)$.  $\S$ is called \emph{balanced} if every circle is positive.  

Two important signatures are the all-positive one, denoted by $+\G=(\G,+)$, and the all-negative one, $-\G=(\G,-)$, where every edge has the same sign.  In most ways an unsigned graph behaves like $+\G$, while $-\G$ acts rather like a generalization of a bipartite graph.  In particular, in $+\G$ every circle is positive.  In $-\G$ the even circles are positive while the odd ones are negative, so $-\G$ is balanced if and only if $\G$ is bipartite.

Signed graphs and balance were introduced by Frank Harary\footnote{Signed graphs, like graphs, have been rediscovered many times; but Harary was certainly the first.  K\"onig \cite[Chapter X]{Konig} had an equivalent idea but he missed the idea of labelling edges by the sign group, which leads to major generalizations; cf.\ \cite[Section 5]{BG1}.} in \cite{NB} with this fundamental theorem:

\begin{thm}[Harary's Balance Theorem]\label{T:balance}
A signed graph $\S$ is balanced if and only if there is a bipartition of its vertex set, $V = X \cup Y$, such that every positive edge is induced by $X$ or $Y$ while every negative edge has one endpoint in $X$ and one in $Y$.  
Also, if and only if for any two vertices $v,w$, every path between them has the same sign.
\end{thm}

A \emph{bipartition} of a set $V$ is any pair $\{X,Y\}$ of complementary subsets, including the possibility that one subset is empty (in which case the bipartition is not, technically, a partition).  I call a bipartition of $V$ as in the Balance Theorem a \emph{Harary bipartition} of $V$.  The Harary bipartition is unique if and only if $\S$ is connected; if $\S$ is also all positive (all edges are positive), then $X$ or $Y$ is empty.

Harary later defined $\S$ to be \emph{antibalanced} if every even circle is positive and every odd circle is negative; equivalently, $-\S$ is balanced \cite{SDual}.  (The negative of $\S$, $-\S$, has signature $-\s$.)  

A basic question about a signed graph is whether it is balanced; in terms of our theme, whether there exists a negative circle.  If $\S$ is unbalanced, any negative circle provides a simple verification (a \emph{certificate}) that it is unbalanced, since computing the sign of a circle is easy.  The Balance Theorem tells us how to provide a certificate that $\S$ is balanced, if in fact it is; namely, one presents the bipartition $\{X,Y\}$, since any mathematical person can easily verify that a given bipartition is, or is not, a Harary bipartition.  What is hard about deciding whether $\S$ is balanced is to find a negative circle out of the (usually) exponential number of circles, or a Harary bipartition out of all $2^{n-1}$ possible bipartitions.  Fortunately, there is a powerful technique that enables us to quickly find a certificate for imbalance.

\emph{Switching} $\S$ consists in choosing a function $\zeta: V \to \{+,-\}$ and changing the signature $\s$ to $\s^\zeta$ defined by $\s^\zeta(e_{vw}) := \zeta(v)\s(e_{vw})\zeta(w)$.  The resulting switched signed graph is $\S^\zeta := (|\S|,\s^\zeta)$.  It is clear that switching does not change the signs of circles.  Let us denote by $\cC(\S)$ the set of all circles of a signed graph (and similarly for an unsigned graph) and by $\cC^+(\S)$ or $\cC^-(\S)$ the set of all positive or, respectively, negative circles.  Thus, $\cC^+(\S^\zeta) = \cC^+(\S)$.  There is a converse due to Zaslavsky \cite{CSG} and, essentially, Soza\'nski \cite{Soz}.

\begin{thm}\label{T:switching}
Let $\S$ and $\S'$ be two signed graphs with the same underlying graph $\G$.  Then $\cC^+(\S) = \cC^+(\S')$ if and only if $\S'$ is obtained by switching $\S$.  In particular, $\S$ is balanced if and only if it switches to the all-positive signed graph $+\G$.
\end{thm}

	\subsubsection*{Algorithmics of balance}\

How do we use this to determine balance or imbalance of $\S$?  Assume $\S$ is connected, since we can treat each component separately.  Find a spanning tree $T$ and choose a vertex $r$ to be its root.  For each vertex $v$ there is a unique path $T_{rv}$ in $T$ from $r$ to $v$.  Calculate $\zeta(v) = \s(T_{rv})$ (so, for instance, $\zeta(r)=+$) and switch $\S$ by $\zeta$.  In $\S^\zeta$ every tree edge is positive.  Every non-tree edge $e$ belongs to a unique circle $C_e$ in $T \cup e$ and $\s(C_e) = \s^\zeta(C_e) = \s^\zeta(e)$.  If there is an edge $e$ that is negative in $\S^\zeta$, then there is a circle $C_e$ that is negative in $\S$ and $\S$ is unbalanced.  If there is no such edge, then $\{X,Y\}$ with $X = \zeta\inv(+) \subseteq V$ and $Y = \zeta\inv(-)$ is a Harary bipartition of $\S$, confirming that $\S$ is balanced.  

Since $T$ can be found quickly by standard algorithms and it is obviously fast to find $\zeta$, this gives us a quick way of determining whether $\S$ is balanced or not.  This simple algorithm was first published (in different terminology) independently by Hansen \cite{Hansen} and then by Harary and Kabell \cite{HK}.

	\subsubsection*{About circles}\

A \emph{chordless} or \emph{induced circle} is a circle $C$ that is an induced subgraph.  Any extra induced edge besides $C$ itself is considered a \emph{chord} of $C$.  

An unsigned graph has girth, $g(\G) = \min_C |C|$, minimized over all circles $C$.  It also has (though less frequently mentioned) even girth and odd girth, where $C$ varies over circles of even or odd length.  These quantities are naturally signed-graphic.  A signed graph has, besides its girth $g(\S)=g(\G)$, also \emph{positive girth} and \emph{negative girth}, $g_+(\S)$ and $g_-(\S)$, which are the minimum lengths of positive and negative circles; they reduce to even and odd girth when applied to $\S=-\G$.  Girth is not explicit in any of my questions but signed girth may be worth keeping in mind.

	\subsubsection*{Contraction}\

Contracting an edge $e=vw$ with two distinct endpoints (a ``link'') in an ordinary graph means shrinking it to a point, i.e., identifying $v$ and $w$ to a single vertex and then deleting the edge $e$.  In a signed graph $\S$, first $\S$ must be switched so that $e$ is positive.  Then contraction is the same as it is without signs; the remaining edges retain the sign they have after switching.

	\subsubsection*{Balancing edges and vertices}\

A \emph{balancing vertex} is a vertex $v$ of an unbalanced signed graph $\S$ that lies in every negative circle; that is, $\S \setm v$ is balanced.  
A \emph{balancing edge} is an edge $e$ in an unbalanced signed graph such that $\S \setm e$ is balanced; that is, $e$ is in every negative circle.  
An endpoint of a balancing edge is a balancing vertex but there can (easily) be a balancing vertex without there being a balancing edge.

A constructive characterization of balancing vertices is the next proposition.  \emph{Contracting} a negative edge $vw$ that is not a loop means switching $w$ (so $vw$ becomes positive) and then identifying $v$ with $w$ and deleting the edge.

\begin{prop}\label{P:balvert}
Let $\S$ be a signed graph and $v$ a vertex in it.  The following statements about $v$ are equivalent.
  \begin{enumerate}[{\rm (1)}]
  \item $v$ is a balancing vertex.
  	\label{P:balvert:bv}
  \item $\S$ is obtained, up to switching, by adding a negative nonloop edge $vw$ to a signed graph with only positive edges and then contracting $vw$ to a vertex, which is the balancing vertex $v$.
  	\label{P:balvert:contract}
  \item $\S$ can be switched so that all edges are positive except those incident with $v$, and at $v$ there is at least one edge of each sign.
  	\label{P:balvert:vsigns}
  \end{enumerate}
\end{prop}

\begin{proof}
The equivalence of \eqref{P:balvert:bv} with \eqref{P:balvert:contract} is from \cite{VLI}.
The result of contraction in \eqref{P:balvert:contract} is precisely the description in \eqref{P:balvert:vsigns}.  
\end{proof}

	\subsubsection*{Blocks and necklaces}\
 
A \emph{cutpoint} is a vertex $v$ that has a pair of incident edges such that every path containing those edges passes through $v$.  For instance, a vertex that supports a loop is a cutpoint unless the vertex is only incident with that loop and no other edge.  A graph is called \emph{inseparable} if it is connected and has no cutpoints.  A maximal inseparable subgraph of $\G$ is called a \emph{block of $\G$}; a graph that is inseparable is also called a \emph{block}.  A \emph{block of $\S$} means just a block of $|\S|$.  Blocks are important to signed graphs because every circle lies entirely within a block.

An \emph{unbalanced necklace of balanced blocks} is an unbalanced signed graph constructed from balanced signed blocks $B_{1}, B_{2}, \ldots, B_{k}$ ($k\geq2$) and distinct vertices $v_i, w_i \in B_i$ by identifying $v_i$ with $w_{i-1}$ for $i=2,\ldots,k$ and $v_1$ with $w_k$.  To make the necklace unbalanced, before the last step (identifying $v_1$ and $w_k$) make sure by switching that a path between them in $B_{1}\cup B_{2}\cup \cdots\cup B_{k}$ has negative sign.  (All such paths have the same sign by the second half of Theorem \ref{T:balance}, because the union is balanced before the last identification.)
An unbalanced necklace of balanced blocks is an unbalanced block in which each $v_i$ is a balancing vertex and there are no other balancing vertices.  
If a $B_i$ has only a single edge, that edge is a balancing edge.  In fact, any signed block $\S$ with a nonloop balancing edge $e$ is an unbalanced necklace of balanced blocks:  the balancing edge is one of the $B_i$'s, and the others are the blocks of $\S \setm e$.  
Unbalanced necklaces of balanced blocks are important in signed graphs; for instance, they require special treatment in matroid structure \cite{BG2}.

If we allow $k=1$ in the definition of a necklace we can say that any signed block with a balancing vertex is an unbalanced necklace of balanced blocks.

%%%%%%%%
\subsection{Parity}\label{parity}\

There is a close connection between negative and positive circles in signed graphs on the one hand, and on the other hand odd and even circles in unsigned graphs---that is, parity of unsigned circles.  

First, parity is what one sees when all edges are negative, or (with switching) when the signature is antibalanced.  There is considerable literature on parity problems that can be studied for possible generalization to signed graphs; I mention some of it in the following sections.  The point of view here is that parity problems about circles are a special case of problems about signed circles.  
Some existing work on odd or even circles will generalize easily to negative or positive circles.  For example, the computational difficulty of a signed-graph problem cannot be less than that of the specialization to antibalanced signatures---that is, the corresponding parity problem---and this may imply that the two problems have the same level of difficulty.

	\subsubsection*{Negative subdivision}\label{negsub}\

Second, there is \emph{negative subdivision}, which means replacing a positive edge by a path of two negative edges.  Negatively subdividing every positive edge converts positive circles to even ones and negative circles to odd ones.  Many problems on signed circles have the same answer after negative subdivision.  The point of view here is that those signed-circle problems are a special case of parity-circle problems.  

Negative subdivision most obviously fails when connectivity is involved since the subdivided graph cannot be 3-connected.  Another disadvantage is that contraction of edges makes sense only in signed graphs; a solution that involves contraction should be done in the signed framework.

Denote by $\S^\sim$ the all-negative graph that results from negatively subdividing every positive edge.  Let $\tilde e$ be the path of length 1 (if $\s(e)=-$) or 2 (if $\s(e)=+$) in $\S^\sim$ that corresponds to the edge $e \in E(\S)$, and for a negative $e$ let $v_e$ be the middle vertex of $\tilde e$; thus, $V(\S^\sim)=V(\S) \cup \{v_e : e \in E^+(\S)\}$.  

The essence of negative subdivision is the canonical sign-preserving bijection between the circles of $\S$ and those of $\S^\sim$, induced by mapping $e \in E(\S)$ to $\tilde e$ in $\S^\sim$.   (There is such a bijection for every choice of positive edges to subdivide, even if that is not all positive edges.)

\begin{prop}\label{T:negsub-bal}
A signed graph $\S$ is balanced if and only if $|\S^\sim|$ is bipartite.  
\end{prop}

\begin{proof}
It follows from the sign-preserving circle bijection that $\S$ is balanced if and only if $\S^\sim$ is balanced.  Since $\S^\sim$ is all negative, it is balanced if and only if its underlying graph is bipartite.
\end{proof}

%%%%%%%%
\subsection{Weirdness}\label{terms}\

	\subsubsection*{Groups or no group}\
	
Any two-element group will do instead of the sign group.  Some people prefer to use the additive group $\bbZ_2$ of integers modulo $2$, which is the additive group of the two-element field $\bbF_2$.  This is useful when the context favors a vector space over $\bbF_2$.  
	
Another variant notation is to define a signed graph as a pair $(\G,\S)$ where $\S \subseteq E(\G)$; the understanding is that the edges in $\S$ are negative and the others are positive.  I do not use this notation.

	\subsubsection*{Terminologies}\
	
Switching has been called ``re-signing'' and other names.

Stranger terminology exists.  Several otherwise excellent works redefine the words ``even'', ``odd'', and ``bipartite'' to mean positive, negative, and balanced, all of which empty those words of their standard meanings and invite confusion.  I say,  ``That way madness lies'' \cite{Lear}.

%%%%%%%%%%%%%%%%%%%%%%%%%%%%%
%%%%%%%%%%%%%%%%%%%%%%%%%%%%%
\section{Say No to Frustration:  Eliminating Negative Circles}\label{frust}

A main question in signed graph theory is how to make an unbalanced signed graph balanced---that is, how to eliminate all negative circles---by adjusting the graph.  Usually, that means deleting edges or vertices, and in particular deleting the smallest number.  The \emph{frustration index} $l(\S)$ is the smallest number of edges, and the \emph{frustration number} $l_0(\S)$ is the smallest number of vertices, whose deletion results in a balanced signed graph.  When $\S$ is antibalanced, i.e.\ (for practical purposes) an all-negative graph $-\G$, then $l(-\G) = |E| - \text{maxcut}(\G) =$ the complement of the maximum cut size, so the frustration index problem is equivalent to the maximum cut size, which is also the maximum number of edges in a bipartite subgraph.  Also, $l_0(-\G) = |V| - \beta(\G)$ where $\beta(\G)$ denotes the maximum order of a bipartite induced subgraph.  Thus, frustration index and number generalize the problems of largest bipartite subgraphs or induced subgraphs.

In general finding the frustration number or index is hard, and finding the maximum value over all signatures of a fixed graph is also hard (see Akiyama et al.\ \cite{Aki}).  An exception is $K_n$, where we have a formula, not very difficult but not too easy:

\begin{prop}[Petersdorf \cite{Pet}]\label{P:petersdorf}
$\max_\s l(K_n,\s) = l(-K_n) = \lfloor (n-1)^2/4 \rfloor$, and if $(K_n,\s)$ is not antibalanced, then $l(K_n,\s) < l(-K_n)$.
\end{prop}

It is easy to verify the analog for frustration number: 
$$
\max_\s l_0(K_n,\s) = l_0(-K_n) = n-2,
$$ 
and if $(K_n,\s)$ is not antibalanced, then $l_0(K_n,\s) < l_0(-K_n)$.

A good theoretical formula for the frustration index is 
\begin{equation}\label{E:swfrust}
l(\S) = \min_\zeta |E^-(\S^\zeta)|,
\end{equation}
minimized over all switching functions $\zeta$.  For computing $l(\S)$ this is impractical because it requires checking an exponential number of switchings ($2^{|V|-1}$, to be exact).  Hence the need for clever methods.  This matters because frustration index is a main question in algorithmic graph theory (for all-negative $\S$; a key word is ``bipartization'') and a significant one in statistical physics.  Both index and number are NP-hard problems (see, e.g., Barahona \cite{Barahona} and Choi, Nakajima, and Rim \cite{CNR}, respectively) but there is much interest in fast algorithms for finding or approximating them.  

In particular, in the ``$\pm J$ Ising model'' in physics fast computation of $l(\S)$ is necessary for computational analysis of examples (see papers of Vogel et al.\ such as \cite{Chileans}, Hartmann such as \cite{Hartmann}, or many other writers).  Present techniques are not strong enough to analyze large graphs.  Since we will not solve that problem, and since this is where I found the term ``frustration'', I only compare terminology.  A ``lattice'' in physics may be a lattice graph or any graph; a ``site'' is a vertex, a ``bond'' is an edge, a ``ferromagnetic bond'' is a positive edge and an ``antiferromagnetic bond'' is a negative edge.  A ``plaquette'' is, while not precisely defined, a kind of chordless circle such that all plaquettes (usually) generate the binary cycle space of the underlying graph.  A ``state'' is a switching function $\zeta$ and an edge $uv$ is ``satisfied'' or ``frustrated'' according as $\s^\zeta(uv) = +$ or $-$; more simply, according as $\zeta(u)\zeta(v) = \s(uv)$ or $-\s(uv)$.  A circle is ``frustrated'' if its sign is $-$, otherwise ``satisfied''.  An unbalanced signed graph is sometimes also called ``frustrated''.

	\subsubsection*{Frustration and negative circles}\

I have mentioned frustration because there are interesting papers on the connection between the frustration index or number and the existence of disjoint negative (or positive) circles.  For instance, the maximum number of edge-disjoint negative circles in $\S$ is at most $l(\S)$ and the maximum number of vertex-disjoint negative circles is at most $l_0(\S)$.  Berge and Reed proved that, if in $\S=-\G$ the maximum number of edge-disjoint circles equals $l(-\G)$, then $\G$ has chromatic number $\chi(\G)\leq3$ \cite{BRodd}.  One could be forgiven for hoping this is a special case of a signed-graph theorem and setting out to prove that theorem.

I will have more to say about packing problems like this in Section \ref{p-c}.

	\subsubsection*{Negative subdivision vs.\ frustration}\

Frustration is not altered by negative subdivision.

\begin{prop}\label{T:negsub-frust}
Negative subdivision changes neither $l_0(\S)$ nor $l(\S)$.
\end{prop}

\begin{proof}
Suppose $F$ is a set of $l(\S)$ edges such that $\S \setm F$ is balanced.  In $\S^\sim$ construct $F^\sim$ by taking each negative edge $e\in F$ and one edge in $\tilde e$ for each positive edge in $F$.  Since $\S \setm F$ is balanced, $\S^\sim \setm F^\sim$ is also balanced; thus, $l(\S^\sim)\leq|F|=l(\S)$.  
Conversely, suppose $G$ is a set of $l(\S^\sim)$ edges in $\S^\sim$ such that $\S^\sim\setm G$ is balanced.  
If $G$ contains one of the edges of a 2-path $\tilde e$ resulting from subdividing a positive edge $e$, it does not contain the other, since that would not eliminate any more negative circles.  Therefore the set $G' = \{e \in E(\S): \tilde e \in G\}$ has cardinality $|G|$.  As $\S^\sim\setm G$ is balanced, so is $\S\setm G'$; thus, $l(\S) \leq |G'| = l(\S^\sim)$.  This proves equality for the frustration index.

Suppose now that $X$ is a set of $l_0(\S)$ vertices such that $\S \setm X$ is balanced.  Then $\S^\sim\setm X$ can have no negative circles, so $l_0(\S^\sim) \leq |X| = l_0(\S)$.  Conversely, suppose $Y$ is a set of $l_0(\S^\sim)$ vertices such that $\S^\sim \setm Y$ is balanced.  If $Y$ contains a vertex $v_e$ where $e$ is a positive edge in $\S$, replace it by an endpoint of $e$ (which is a vertex of $\S^\sim$).  That gives a vertex set $Y' \subseteq V(\S)$ such that $|Y'| \leq |Y|$ and $\S \setm Y'$ is balanced, so $l_0(\S) \leq |Y'| \leq l_0(\S^\sim)$.  That proves equality for the frustration number.
\end{proof}

This implies that the frustration index or number of a signed graph can be computed by any algorithm that computes the \emph{bipartization index} or \emph{number} of an unsigned graph, which is the minimum number of edges or vertices, respectively, whose deletion makes an unsigned graph bipartite.  As bipartization is the all-negative case of frustration (and negative subdivision can obviously be computed in linear time), the two problems are equivalent in computational difficulty.  (The considerable effort that has been devoted by some physicists to speeding up computation of $l(\S)$ \cite{Hartmann,Chileans} is therefore equivalent to speeding up the calculation of bipartization index.  Physicists are usually more interested in particular kinds of graphs, such as lattice graphs, for which there may be special methods of computation that do not permit subdividing edges.)

%%%%%%%%

%%%%%%%%%%%%%%%%%%%%%%%%%%%%%
%%%%%%%%%%%%%%%%%%%%%%%%%%%%%
\section{Edges and Vertices in Negative (and Positive) Circles and Holes}\label{ev}

\begin{enumerate}[\Q $1^-$.]
\item  In $\S$, is a certain edge $e$ in a negative circle?
\\\A  It depends on the block containing $e$; see Theorem \ref{T:e-posneg}(a).  (Easy.)
\label{NC}
\item  In $\S$, is a certain edge $e$ in a unique negative circle?
\\\A  It depends on the Tutte 3-decomposition of $\G$ into 2-connected subgraphs, and the details of how those subgraphs are signed.  (Medium hard; solved by Behr \cite{Behr}; see Section \ref{S:e-unique}.)
\label{UNC}
\item  In $\S$, find all edges $e$ such that $e$ belongs to a unique negative circle.
\\\A  This is essentially the same as Question \ref{UNC}.
\item  In $\S$, is a certain edge $e$ in a negative chordless circle?
\\\A  Unknown.  There is a recent algorithm by Marinelli and Parente \cite[especially Section 4.2]{MarPar} but no study of optimality and no answer in terms of graph structure.
\label{NH}
\item  In $\S$, is a certain edge $e$ in a unique negative chordless circle?
\\\A  Unknown.
\label{UNH}
\item  In $\S$, find all edges $e$ such that $e$ belongs to a unique negative chordless circle.
\\\A  This is essentially another version of Question \ref{UNH}.
\end{enumerate}

The analogous questions for {positive circles} are Questions $n^+$.  

The analogous questions for a vertex are easily answered from the answers for edges, because a vertex belongs to a circle if and only if some incident edge belongs to that circle.

All these questions are reducible by negative subdivision to parity questions in unsigned graphs.  However, it may well be easier to go the other way: answer them for signed graphs, then specialize to antibalance to get parity corollaries.

%%%%%%%%
\subsection{An edge (or vertex) in a circle of specified sign}\label{S:e-posneg}\

Whether an edge $e$ is in a negative circle, or a positive circle, clearly depends only on the block that contains $e$.  The answer is easy to prove and nicely illustrates the use of Menger's Theorem.  (I do not say these results are new, though I do not know a source for \eqref{e-posneg:pos}.)  Curiously, the conditions for an edge to be in a positive circle are the more complicated.

\begin{thm}\label{T:e-posneg}
Let $\S$ be a signed graph with an edge $e$.
\begin{enumerate}[{\rm(1)}]
\item There is a negative circle that contains $e$ if and only if the block containing $e$ is unbalanced (Harary \cite{LB}).%
\label{e-posneg:neg}
\item There is a positive circle that contains $e$ if and only if $e$ is not an isthmus and either the block containing $e$ is balanced or it is unbalanced and $e$ is not a balancing edge of the block.
\label{e-posneg:pos}
\end{enumerate}
\end{thm}

\begin{proof}
We may assume $\S$ is a block.

If $\S$ is balanced, there are no negative circles, and if $\S \neq K_2$ then every edge is in a circle and every circle is positive.  That solves the balanced case.

Assume $\S$ is unbalanced so there is a negative circle $C$.  Suppose $C$ can be chosen so it does not contain $e$.  Then $e$ is not a loop so it has distinct endpoints $u, v$.  By Menger's Theorem (see, e.g., \cite[Theorem 3.3.1]{Diestel} for the right version of that multifaceted theorem)
there exist disjoint paths $P$ from $u$ to $C$ and $Q$ from $v$ to $C$ that are internally disjoint from $C$.  The union $C \cup \{e\} \cup P \cup Q$ is a theta graph whose two circles containing $e$ have opposite signs.  Thus, $e$ is in both a positive and a negative circle.

If $C$ cannot be chosen so it does not contain $e$, then $e$ is in every negative circle so it is a balancing edge of $\S$.  Clearly, $e$ is then in some negative circle.  On the other hand, $\S \setm e$ is balanced so $\S$ can be switched to make $E \setm e$ all positive; then $e$ is negative, so it is clear that every circle containing $e$ is negative.
\end{proof}

\begin{cor}\label{C:v-posneg}
Let $\S$ be a signed graph with a vertex $v$.
\begin{enumerate}[{\rm(1)}]
\item There is a negative circle that contains $v$ if and only if $v$ belongs to an unbalanced block.
\item There is a positive circle that contains $v$ if and only if $v$ belongs to a balanced block that is not $K_1$ or $K_2$ or it belongs to an unbalanced block in which it is not a balancing vertex of degree $2$.
\end{enumerate}
\end{cor}

%%%%%%%%
\subsection{An edge in a unique negative (or positive) circle}\label{S:e-unique}\

Behr's solution to Question $\ref{UNC}^-$ illustrates the role of structural graph theory, in particular the structure of 2-separations, in solving signed circle problems.  Clearly, it is enough to answer the question for a signed block.  

In a graph $\G$ with a subgraph $\Delta$, consider a maximal subgraph $B$ such that every vertex of $B$ is connected to every other by a path that is internally disjoint from $\Delta$.  We call $B$ a \emph{bridge of $\Delta$} (cf.\ Tutte's textbook \cite{Tbook}) and the vertices in $V(B) \cap V(\Delta)$ the \emph{vertices of attachment} of $B$.  Bridges are fundamental in structural graph theory; bridges of a circle are essential to questions about negative or positive circles in signed graphs.  

Suppose $\G$ is a block and $\Delta$ is a circle $C$ (and not a loop); then $B$ has at least two vertices of attachment.  If it has only two, we call it a \emph{path bridge} (but we do not require $B$ itself to be a path).  
If two path bridges $B_1$ and $B_2$ have attachment vertex pairs that separate each other along $C$ (that means $B_1$ is attached at $v_1, w_1$ and $B_2$ at $v_2, w_2$ and these vertices appear along $C$ in the order $v_1v_2w_1w_2$, no two being equal), we say $B_1$ and $B_2$ are \emph{crossing bridges}.  
If $B$ is a bridge of $C$ with attachment vertices $v,w$ such that one of the two segments of $C$ connecting $v$ and $w$ contains no other vertices of attachment, we call that segment of $C$ a \emph{handle} of $C$.  

\begin{thm}[from \cite{Behr}]\label{T:behr}
Let $\S$ be a signed block.  An edge $e$ is contained in a unique negative circle if and only if either $\S$ itself is a negative circle and $e$ is any edge, or $\S$ properly contains a negative circle $C$ such that the bridges of $C$ are non-crossing path bridges, $C$ has exactly two handles, $\S$ has a balancing edge that belongs to one handle of $C$, and $e$ belongs to the opposite handle.
\end{thm}

Under the conditions of the theorem, the balancing edges of $\S$ are all the edges of the handle that contains a balancing edge.  The proof depends on showing that $e$ belongs to more than one negative circle if the bridge conditions are not satisfied.  
The same proof solves the complementary Question $\ref{UNC}^+$ (see \cite{Behr}).

%%%%%%%%
\subsection{An edge (or vertex) in an induced circle of specified sign}\label{S:e-posneg-ind}\

This problem is harder than the previous ones.  Because it depends on induced circles, and because by subdividing every edge we can make every circle induced, the structural approach, independent of subdivision, that works for circle problems cannot be applied.  
On the other hand, consider a vertex $v$ in a triangle-free signed graph.  By adding suitably signed edges in the neighborhood $N(v)$ we can ensure that all the induced circles containing $v$ will be triangles of either desired sign, regardless of the rest of $\S$.  Similar remarks apply to an edge.

Consider the opposite extreme to subdivision: signed complete graphs $(K_n,\s)$, where every vertex has high degree and every induced circle is a triangle.  It is easy to test whether vertex $v$ belongs to a negative triangle:  first switch to make all edges incident with $v$ positive; then examine $N(v)$ (which in this example is $\S \setm v$) to see if it contains a negative, or positive, edge---that tells whether or not $v$ belongs to a negative, or positive, induced circle.  To test an edge $vw$, compare positive and negative neighborhoods, defined as $N^+(u) = \{ x \in N(u): \s(ux)=+ \}$ and similarly $N^-(u)$.  The edge $vw$ belongs only to positive triangles if and only if $N^+(v) = N^+(w)$ and only to negative triangles if and only if $N^+(v) \cap N^+(w) = \eset.$

Contemplating these examples, I suspect that a good answer to Questions $\ref{NH}^\pm$ will have to be algorithmic (happily, just what was wanted by Marinelli and Parente).  The problem is to find a relatively good algorithm.  The traditional first question is whether there is a polynomial-time algorithm, and even before that, whether the problem belongs to the class NP.  

\begin{prop}\label{P:e-pni-NP}
The question, ``Does a given edge $e$ in a signed graph $\S$ belong to some negative induced circle?'', is in the class NP, and so is the same question for a positive induced circle.
\end{prop}

\begin{proof}
A quickly verifiable certificate that $e$ is in a negative induced circle is the circle.  The verification that it is a circle, has no chords, and is negative (or positive), are all fast.
\end{proof}

I do not know whether these questions are polynomial-time solvable, NP-complete, or in between (given the usual caveat about the unproved difference between P and NP).  By crude analogy with frustration index and frustration number, I expect both are NP-complete.  

Then, there is the opposite question.

\begin{qn}\label{Q:e:xpni-NP}
Is the question, ``Does a given edge $e$ in a signed graph $\S$ \emph{not belong to any} negative induced circle?'', in the class NP?---and the same for a positive induced circle.
\end{qn}

An edge $e$ can belong to a circle of a certain sign but no induced circle of that sign.  A circle with a chord gives easy examples.

%%%%%%%%

%%%%%%%%%%%%%%%%%%%%%%%%%%%%%
%%%%%%%%%%%%%%%%%%%%%%%%%%%%%
\section{The Systems of Negative (and Positive) Circles and Holes}\label{struct}

These are questions about the relationships between circles.

A \emph{hole} is a chordless circle of length at least 4 (usually in a simple graph); triangles are excepted because many questions about graphs are answered by excluding holes, or odd or even holes, but not triangles.

\begin{enumerate}[\Q 1.]
\item  Can a given set of circles in $\G$ be the negative circles of a signature?
\\\A  In every theta subgraph, of the three circles, an even number must be in the set.  (Easy \cite{CSG}.)
\label{Q:balcircset}
\item  Can a set of chordless circles in $\G$ be the negative chordless circles of a signature?
\\\A  There is an infinite set of requirements involving subgraphs of a finite number of types.  (Hard: see Truemper \cite{alphabal}.)
\label{Q:alphabal}
\item  Characterize the signed graphs in which any two negative circles have at least one common vertex.  
\\\A  Solved.  (Hard.)  Slilaty \cite{Sli} completed the proof of this classification, which L\'ov\'asz had initiated; see Section \ref{no2}.
\item  Characterize the signed graphs in which any two negative circles have at least two vertices in common.  I call them \emph{quasibalanced}.
\\\A  Soon to be known \cite{QB}; see Section \ref{S:qbal}.  (Moderately hard.)
\label{Q:quasibal}
\item  Characterize the signed simple graphs with no chordless negative circles.  Or, with none other than triangles; i.e., no negative holes.
\\\A  This is easy for chordal graphs---also known as triangulated graphs, because the definition is that every circle longer than a triangle has a chord.  The answer:  If every triangle in a signed chordal graph is positive, the graph is balanced.  In general the questions are hard.  
\label{Q:oddsignable}
\item  Characterize the signed simple graphs with no chordless positive circles.  Or, with none other than triangles.
\\\A  This is also easy for chordal graphs:  If every triangle in a signed chordal graph is negative, the graph is antibalanced.  Again, the general questions are not at all easy.  
\label{Q:evensignable}
\end{enumerate}

Questions \ref{Q:balcircset}--\ref{Q:quasibal} are not affected by negative subdivision, so they can in principle be solved as parity problems.  However, I think that is the wrong way to approach them because the structures seem more visible in the signed-graph view.

As far as I am aware, research on Questions \ref{Q:oddsignable} and \ref{Q:evensignable} has focussed on signed graphs with only negative triangles and negative holes, in the form of the unsigned graphs that have such a signature (called ``odd-signable''), and on signed graphs with only negative triangles and positive holes, also in the form of their underlying graphs (called ``even-signable'').  Vu\v{s}kovi\'c \cite{Vsurvey} surveys even-signable graphs, their structure, and algorithms, and mentions odd-signable graphs.  The fundamental structure theorems for both kinds are from \cite{CCKV}, with many subsequent papers.

Questions involving chordless circles cannot be treated by negative subdivision.  A negative chord ceases to be a chord if it is subdivided; worse, switching can change which chords are negative.  
I wonder exactly how negative subdivision, signed chords, and switching interact.

\begin{enumerate}[\Q 7.]
\item What properties of $\S$ and $C$ imply that a negative (or positive) circle $C$ with one or more chords does or does not become chordless in $\S^\sim$ after switching $\S$.
\end{enumerate}

%%%%%%%%
\subsection{No two disjoint negative circles}\label{no2}\

Consider a series of intersection properties of negative circles.  First, there are signed graphs with non-intersecting negative circles---most signed graphs.  Then there are those in which any two negative circles intersect.  Slilaty \cite{Sli} proved a characterization, of which the main part is the signed graphs that can be embedded in the real projective plane.

A signed graph \emph{embeds} in that plane if it can be drawn without self-intersections so that the positive circles are contractible but the negative circles are not.  No two negative circles can be disjoint because any two noncontractible curves intersect.  These are the principal examples of signed graphs with no two disjoint negative circles; the other basic example is $-K_5$; and then one can attach an arbitrary balanced graph in certain ways.  See \cite[Theorem 1.2]{Sli}.  

Hochst\"attler et al.\ \cite{HNP} have an algorithm to decide the existence of two disjoint negative circles in polynomial time and to find them if they exist.

%%%%%%%%
\subsection{Quasibalance}\label{S:qbal}\

The next step in the series of intersection properties is quasibalance.  
In the frame and lift matroids of a signed graph \cite{SG, BG1} there are two kinds of matroid circuit: positive circles, and certain subgraphs that contain two negative circles with at most one common vertex.  Quasibalanced signed graphs are those in which the latter type does not occur.  (That is how the question of quasibalance first arose \cite{Sconn}.)  
The next property in the series is that every pair of negative circles has at least three common vertices, but at present I know of no reason to be interested in such graphs.

I will now describe a reduction of Question \ref{Q:quasibal}.  An easy lemma reduces the problem to blocks.

\begin{lem}\label{L:qbal:blocks}
A signed graph is quasibalanced if and only if it has at most one unbalanced block, which is itself quasibalanced.
\end{lem}

As a preliminary classification of quasibalanced signed blocks, each falls into exactly one of the following types.
  \begin{enumerate}[{\quad \rm (a)}]
  \item  Balanced.  
  \label{L:qbal:bal}
  \item  Unbalanced, with two (or more) balancing vertices.  
  \label{L:qbal:2bv}
  \item  Unbalanced and quasibalanced but with no balancing vertex.  
  \label{L:qbal:nobv}
  \item  Unbalanced and quasibalanced with only one balancing vertex.  
  \label{L:qbal:1bv}
  \end{enumerate}
It is not obvious that the third and fourth types exist; indeed, the fourth does not.  The third does: a few examples are $K_{4}$ with all edges negative or equivalently (by switching) with a negative 2-edge matching, and $K_{3,3}$ with a negative 2-edge or 3-edge matching.  The second exists and can be described fully.  

\begin{prop}\label{P:necklace}
A signed block has two (or more) balancing vertices if and only if it is an unbalanced necklace of balanced blocks.
\end{prop}

A complete description of the third type is complicated.  There is a structural approach based on bridges of a negative circle (bridges again!);  one can prove its bridges are balanced.  The description of type \eqref{L:qbal:nobv} then depends on how bridges interact.  
The analysis will appear in \cite{QB}.  

%%%%%%%%
\subsection{Beyond quasibalance}\label{S:negint}\

In general, what is the intersection of all negative circles of a signed graph $\S$?  Apply the negative-subdivision trick to $\S$, yielding a graph $\G$.  Apply the fast Cai--Schieber algorithm \cite{CS} to $\G$ and you have the intersection of all negative circles in $\S$.  QED.

%%%%%%%%

%%%%%%%%%%%%%%%%%%%%%%%%%%%%%
%%%%%%%%%%%%%%%%%%%%%%%%%%%%%
\section{Packing and Covering}\label{p-c}

\emph{Covering} $\S$ by signed circles means finding a set of circles of the right sign such that every vertex, or every edge, is in one (or more) of the circles.  One wants to minimize the number of circles in a cover.  If the circles are edge-disjoint we we call the covering a \emph{decomposition} of $\S$.  
\emph{Packing} signed circles means finding circles of the right sign that are vertex- or edge-disjoint.  One wants to maximize the number of such circles, or minimize the number of vertices or edges that are not covered by their union.  A set of edge-disjoint circles is a decomposition if and only if it is both a packing and a covering.

Packing, covering, and decomposition are natural and popular types of graph-theory problem.  There has been less attention paid them in signed graph theory, perhaps because relatively few graph theorists are yet familiar with signed graphs.

Negative subdivision makes little difference for questions of packing, decomposition, and edge covering, because the circles and the packing and covering properties are not affected by it.  E.g., if $C_1,\ldots,C_k$ cover the edges of $\S$, then $C_1^\sim,\ldots,C_k^\sim$ cover the edges of $\S^\sim$.  Vertex covering is different: if  $C_1,\ldots,C_k$ cover $V(\S)$, $C_1^\sim,\ldots,C_k^\sim$ need not cover all the extra vertices $v_e$ of $\S^\sim$.  Thus, most of the questions in this section can be reduced to odd and even circles in an unsigned graph; but the conjecture and theorem of Huynh et al.\ (Question $10^+$) show that approach may be inadequate.

%%%%%%%
\subsection{Packing circles}\label{pack}\

Let $p(\G)$ and $p'(\G)$ be the maximum number of vertex- or edge-disjoint circles in a packing of $\G$.  The signed analogs are $p_-(\S)$, $p_+(\S)$, $p_-'(\S)$, and $p_+'(\S)$' the subscript denotes the sign of the circles allowed in the packing.

\begin{enumerate}[\Q $1^-$.]

\item  Given $\S$, what is the value of $p_-(\S)$?
\\\A  Unknown.  It is obvious that $p_-(\S) \leq \min( l_0(\S), p(|\S|) )$.  \emph{Subquestion}.  Which signed graphs have equality?  Equality can occur; to create such a signature on $\G$ find a packing of $k \leq p(\G)$ circles in $\G$ and let $E^-$ consist of one edge from each circle in the packing; then $p_-(\S) = k =  l_0(\S)$.  But these are atypical signatures.

Conforti and Gerards \cite{ConfGer} show that evaluating $p_-(\S)$ is NP-hard, but it can be solved in polynomial time if one excludes from $\S$ four switching classes of signed graphs.  This does not answer my subquestion because both $l_0$ and $p$ are NP-hard.

Geelen and Guenin \cite{GeeGue} study the packing problem in Eulerian graphs (the word ``odd'' in their title means negative circles in signed graphs, not odd circles in ordinary graphs).

There is an explicit lower bound for signed \emph{planar} graphs, the best I know of being $p_-(\S) \geq l_0(\S)/6$, by Kr\'al', Sereni, and Stacho \cite{Kral}.  They say this is probably too weak; $p_-(\S) \geq l_0(\S)/2$ may be generally true and is true for ``highly connected'' antibalanced graphs by Thomassen \cite{Tssen}, the required connectivity being more closely evaluated by Rautenbach and Reed \cite{RautReed}.

Another parity paper is Berge and Reed \cite{BR}, with an important result about the antibalanced case (see my remarks in Section \ref{frust}).  
\label{Pnegnumber}

\item[\Q $\ref{Pnegnumber}^+$.]  The same question for $p_+(\S)$.  
\\\A  Unknown.  This looks harder than Question \ref{Pnegnumber}$^-$, as with positive circles there is no known natural upper bound analogous to $l_0(\S)$.  

There is a lower bound for the all-negative case by Chiba et al.\ \cite{Ch}:  there exist at least $k$ vertex-disjoint positive (i.e., even) circles in $-\G$ if every vertex has degree at least $k$, $n\gg ck^{8k}$ (approximately), and $\G$ is not in a short list of exceptions.  This bound leaves something to be desired, as one would usually expect $p_-(-\G)\gg k$ for such large $n$.
\label{Pposnumber}

\item[\Q $1^-i$.]  Golovach et al.\ \cite{GKPT} raise a curious variant of vertex-disjoint packing: the union of the odd circles should be an induced subgraph.  They prove that for planar graphs such an ``induced packing of $k$ odd circles \ldots\ can be found (if it exists) in time $2^{O(k^{3/2})}n^{2+\epsilon}$ (for any constant $\epsilon>0$)''.  
By the negative subdivision trick, the same holds true for signed as well as unsigned planar graphs, since negative subdivision can at most double $n$.
But deciding the existence of an induced packing of only two odd chordless circles in an arbitrary graph is NP-complete.

\item[\Q $2^\pm$.]  Find a maximum set of pairwise disjoint negative, or positive, circles.
\\\A  Unknown.  This is simply a more demanding version of Question \ref{Pnegnumber}.  An answer should be an efficient algorithm.
\label{Pnegset}

\label{Plast}
\end{enumerate}

And, of course, the same questions for edge-disjoint circles:

\begin{enumerate}[\Q $3^-$.]

\item  That $p_-'(\S) \leq \min( l(\S), p'(|\S|) )$ is obvious.  When is there equality?
\\\A  Unknown.  Examples with $p_-'(\S) = k = l(\S)$ for any $k \leq p'(\G)$ can be created on any graph in the same way as for Question \ref{Pnegnumber}$^-$.  

\emph{Conjecture}.  There is always equality.  I found this to be true for $K_3$, $K_4$, and $K_5$.  Proposition \ref{P:packsmall} is (weak) further support.  
On the other hand, Kr\'al' and Voss's bound for planar graphs \cite{KV} suggests the conjecture may be wrong.
They proved, assuming $|\S|$ is planar, that $p_-'(\S) \geq l(\S)/2$ with cases of equality.  I am not sure what that implies for my conjecture.

\item[\Q $3^+$.]  Evaluate $p_+'(\S)$, given $\S$.
\\\A  Unknown.  As with $p_+(\S)$, there is no known positive analog of the upper bound $l(\S)$ to suggest an answer.

\item[\Q $4^\pm$.]  Find a maximum set of pairwise edge-disjoint negative, or positive, circles.
\\\A  Unknown.\end{enumerate}

Here is a verification of Conjecture $3^-$ for signed complete graphs when the frustration index is small.  Not so incidentally, the packing number of triangles in $K_n$ is known; see Feder and Subi \cite{FedSu}.

\begin{prop}\label{P:packsmall}
If the frustration index of $(K_n,\s)$ satisfies $l(K_n,\s) \leq (n-1)/2$, then $p_-'(K_n,\s) = l(K_n,\s)$.
\end{prop}

\begin{proof}
For $n\leq4$ this is trivial or obvious.  Consider any other $(K_n,\s)$; assume by switching that the number of negative edges is $l=l(K_n,\s)$.  The negative edge set $E^-$ consists of one or more components, $E_i^-=\{e_{i,1},\ldots,e_{i,l_i}\}$ for $1 \leq i \leq m \leq l$, having $l_i$ edges and $n_i$ vertices with $2 \leq n_i \leq l_i+1$ and equality only if $E_i^-$ is a tree.  
More precisely, $n_i = l_i + 1 - \xi_i$, where $\xi_i$ is the cyclomatic number of $E_i^-$.  (The cyclomatic number is the number of edges not in a maximal forest.)  Note that $\xi(E^-)=\sum_i \xi_i$.

The simple trick is to create a negative triangle containing $e \in E_i^-$ by joining it to a vertex not in $V(E_i^-)$.  The difficulty is to ensure that no positive edge is used twice.  We ensure this by using a different third vertex for every negative edge.  Thus, we have to demonstrate that there are enough vertices available for making negative triangles.

The number of vertices not in negative edges (call them extra vertices) is 
\begin{equation}
\begin{aligned}
n-\sum_i n_i &= n-\sum_i (l_i+1-\xi_i) = n - (l+m-\xi(E^-)) \\
&= (n - 2l) + \sum_i (l_i-1) + \xi \\
&\geq \sum_{i=1}^{m-1} (l_i-1) + l_m.
\end{aligned}
\label{E:packsmall}
\end{equation}
For each $E_i^-$ with $i<m$ we choose one vertex $v_{i,l_i}\in V(E_{i+1}^-)$ and $l_i-1$ extra vertices $v_{i,1},\ldots,v_{i,l_i-1}$, and for $E_m^-$ we choose $l_m$ extra vertices $v_{m,1},\ldots,v_{m,l_m}$, so that no extra vertex is chosen twice; Equation \eqref{E:packsmall} shows there are enough extra vertices to do that.  The triangles on $V(e_{i,j}) \cup \{v_{i,j}\}$ for $j=1,\ldots,l_i$ each have exactly one negative edge, and no two have an edge in common.  Therefore we have a packing of $l$ negative circles, proving that $p_-(K_n,\s)\geq l(K_n,\s)$.  Since $p_-'(K_n,\s) \geq l(K_n,\s)$ is always true, the proof is complete.
\end{proof}

The upper limit $(n-1)/2$ can certainly be raised, probably to around $n$.  I used third vertices for triangles on negative edges very inefficiently.  My proof leaves at least $\xi(E^-)$ unused extra vertices; I could have used two or more vertices in $E_{i+1}^-$ instead of extra vertices (if $m>1$); and especially I could have used the same third vertex more than once.  Besides all that, the negative circles in the packing need not be triangles; for instance, the three negative circles that pack $K_5$ are two triangles and one quadrilateral.  
I present improvement of Proposition \ref{P:packsmall} as an open problem.

%%%%%%%
\subsection{Covering by circles}\label{cover}\

I cannot recall seeing any papers on covering, not even for the graphic case where one asks for odd or even circles, i.e., where $\S=-\G$ is all negative.

\begin{enumerate}[\Qs 1--2$^\pm$.]

\item[\Qs 5--6$^\pm$.]  Like Questions 1--2$^\pm$ but for the minimum number or minimum sets of negative (or positive) circles that cover all the vertices of $\S$.
\\\A  Unknown.

\item[\Qs 7--8$^\pm$.]  Like Questions 5--6$^\pm$ but for circles that cover the edges of $\S$.
\\\A  Unknown.

\item[\Q 9$^\pm$.]  Are there duality relations between packing and covering numbers?
\\\A  Unknown.  

\label{Clast}
\end{enumerate}

%%%%%%%
\subsection{Decomposition into circles}\label{decomp}\

These problems are suggested by the theorem that a connected graph decomposes into circles iff it is Eulerian.  (\emph{Decomposing} a graph means partitioning its edge set.)  Questions 10--$11^\pm$ seem very hard but interesting since the antibalanced case $-\G$ is asking for decomposition into odd, or even, circles.
Let $d(\G)$ denote the smallest number of circles into which $\G$ can be decomposed.

\begin{enumerate}[\Q $10^\pm$.]

\item[\Q $10^-$.]  Which $\S$ can be decomposed into negative circles?
\\\A  Unknown.  
\label{Dnegdecomp}

\item[\Q $10^+$.]  Which $\S$ can be decomposed into positive circles?
\\\A  Partially known.  The best current result is due to Huynh, Oum, and Verdian-Rizi \cite{HOV}.  FIrst, their exciting \emph{Conjecture}.  A connected signed graph $\S$ has a decomposition into positive circles if and only if it has even degree at every vertex, it has an even number of positive edges (these are obvious), and it does not have a subgraph that contracts to $-K_5$ (this is the subtle part).   What they prove is sufficiency of the condition with $-K_4$ replacing $-K_5$ and another small restriction.
Earlier, M\'a\v{c}ajov\'a and Maz\'ak \cite{MM} found an infinite family of signed graphs that are 4-regular (so they do have a circle decomposition) but have no such decomposition into positive circles.

In \cite{HKOV} the same authors and King study the property of \emph{strong circle decomposability} of a graph $\G$: every subdivision of $\G$ with an even number of edges decomposes into even circles.  They treat this property through signs on $\G$.
\label{Dposdecomp}

\item[\Q $11^-$.]  If $\S$ can be decomposed into negative circles, what is the smallest number of circles it needs?
\\\A  Unknown.  The answer is clearly $\geq d(|\S|)$ and $\leq l(\S)$, so there can be no negative circle decomposition if $l(\S) < d(|\S|)$.  I found that every signed $K_5$ with $l(K_5,\s) \geq d(K_5) = 2$ (they all have $l(K_5,\s)=2$ or $3$) has a decomposition into $l(K_5,\s)$ but no fewer negative circles.
\label{Dnegdecompnum}

\item[\Q $11^+$.]  If $\S$ can be decomposed into positive circles, what is the smallest number of circles?
\\\A  Unknown.  The number is $\geq d(|\S|)$, but when it may be equal, and how much greater it can be, are unknown.
\label{Dposdecompnum}

\item[\Q $12^\pm$.]  Is there an interesting property of a connected signed graph, similar to existence of an Eulerian tour, related to decomposition into negative or positive circles?
\\\A  Unknown.  Needless to say, this question is open-ended.  
\label{Dposnegdecomp}

Research on decomposing unsigned graphs into even circles (that is, circles of even length) led Huck and Kochol \cite{HK} to broaden the question by introducing an intermediate kind of decomposition and a nice parameter they called ``oddness'' of the graph.  This naturally extends to signed graphs, suggesting this intermediate problem that enlarges the perspective of positive-circle decomposition.  Define the \emph{circle negativity} of $\S$ to be the smallest number of negative circles in any circle decomposition of the underlying graph.

\item[\Q $13^+$]  If $\S$ can be decomposed into circles (that is, if all degrees are even), what is its circle negativity?

There is the obvious complementary question about the circle positivity of $\S$.  For unsigned graphs (that is, all-negative signed graphs) that seems less interesting, but for signed graphs in general, bearing in mind the essentiality of negative circles, it should be interesting to look for the fewest positive circles in a circle decomposition.  Thus, I propose:

\item[\Q $13^-$]  If $\S$ can be decomposed into circles (that is, if all degrees are even), what is its circle negativity?

Perhaps the last two questions are the most interesting ones!

\label{Dlast}
\end{enumerate}

%%%%%%%%

%%%%%%%%%%%%%%%%%%%%%%%%%%%%%
%%%%%%%%%%%%%%%%%%%%%%%%%%%%%
\section{Structural Circle Questions}\label{neg}

An assortment suggested partly by existing graph theorems and the popularity of Hamiltonian questions.

\begin{enumerate}[\Q 1.]

\item  Assume $\S$ has a Hamiltonian circle and is unbalanced.  Is there a negative Hamiltonian circle?  A positive one?
\\\A  Unknown.  

\emph{Conjecture}.  Most $\S$ with a Hamiltonian circle have both signs.  The exceptions are the balanced signed graphs and the antibalanced signed graphs of even order, which can have only positive Hamiltonian circles, as well as the antibalanced signed graphs of odd order and the unbalanced necklaces of balanced blocks, which can have only negative Hamiltonian circles.

Popescu \cite{Pop1a} proved that if $(K_n,\s)$ is neither balanced nor antibalanced, then it has both positive and negative circles of all lengths.  In particular it has both positive and negative Hamiltonian circles, but Popescu's result suggests a bigger question:

\item  For which graphs $\G$ is it true that every signature $\s$ has both positive and negative circles of every length that occurs in $\G$?
\\\A  I know of nothing other than Popescu's theorem.

\item  Is there a positive, or negative, circle $C$ such that $\S \setminus E(C)$ is disconnected, or separable, or 2-separable, or 2-connected?
\\\A  Conlon \cite{Conlon} proved that if $\G$ is 3-connected, there is an even circle $C$ such that $\G \setminus E(C)$ is 2-connected.  
Fujita and Kawarabayashi \cite{FuKawara} have a similar theorem for $\G \setminus V(C)$.
Do these generalize to signed graphs, evenness generalizing to positivity?  What definition of connectivity of a signed graph is suggested?

\item  What are the bridges (in the sense of Tutte) of a negative or positive circle?  For instance, does the circle have many chords?
\\\A  Voss \cite{Voss} studied chords and other properties of circles in $\G$ of given parity.  Which of these generalize to circles of given sign in $\S$?

In general the bridges of a circle can be anything.  This question should be asked of signed graphs of special kinds.  An example is quasibalanced signed graphs (Section \ref{S:qbal}), in which a bridge of a negative circle must be balanced.

\label{Slast}
\end{enumerate}

%%%%%%%%

%%%%%%%%%%%%%%%%%%%%%%%%%%%%%
%%%%%%%%%%%%%%%%%%%%%%%%%%%%%
\section{Counting Negative Circles}\label{count}

The \emph{negative circle vector} is $c^-(\S) = (c_1^-,c_2^-,\ldots,c_n^-) \in \bbR^{n}$, where $n=|V|$ and $c_k^-(\S)$ is the number of negative circles of length $k$.  These numbers and vectors have had some attention, mostly aimed at underlying complete graphs.  I distinguish two types of question: about numbers, and about vectors.

\begin{enumerate}[\Q 1.]

\item  Characterize the set of numbers of negative circles of some fixed length of all signatures of a fixed graph $\G$; that is, the set $\mathbf C_k(\G) = \{c_k^-(\G,\s) : \s$ is a signature of  $\G\}$ for some fixed $k$, $3\leq k \leq n$.
\\\A  Only $\G=K_n$ has been studied, as far as I know.  Very recently there are remarkably strong results on the possible numbers of negative triangles by Kittipassorn and M\'esz\'aros \cite{KittiMesz}.  Two-thirds of the numbers from $0$ to $\binom{n}{3}$ cannot appear.  
There is a function $f(n)$ such that $[f(n),\binom{n}{3}-f(n)] \subseteq \mathbf C_3(K_n)$ for $n\gg0$.  
And more, especially:

\begin{thm}[\cite{KittiMesz}]\label{T:km}
Let $0=a_0\leq b_0 \leq \cdots \leq a_m \leq b_m \approx n^{3/2}$ where $a_{i+1}=b_i+(n-2)-i(i+1)$ and $b_i-a_i=i(i-1)$; then $a_i,b_i\in\mathbf C_3(K_n)$; also, for $0\leq i\leq m$ we have $c_3^-(K_n,\s)\in[a_i,b_i] \iff l(K_n,\s)=i$.  
\end{thm}

I am not aware of similarly strong conclusions about longer circles, but several papers by Popescu and Tomescu have partial results.  
For example, all $c_k^-(K_n,\s)>0$ if $\s$ is neither balanced nor antibalanced \cite{Pop1}.  Also:

\begin{thm}[Popescu and Tomescu \cite{PopTom1}]\label{T:ptmin}
Fix $s \leq n/2$ and consider only signatures for which $l(K_n,\s)=s$; then for all lengths $k$, $\min_\s c_k^-(K_n,\s)$ is attained when $E^-$ is a star (if $s<n/2$) and $\max_\s c_k^-(K_n,\s)$ is attained when $E^-$ is a matching \cite{PopTom1}.  
\end{thm}

Their original theorem assumed $|E^-|=s$ instead of $l(K_n,\s)=s$.  This restatement depends on a lemma:

\begin{lem}\label{L:smallneg}
If $|E^-(K_n,\s)|<n/2$, then $l(K_n,\s)=|E^-|$.
\end{lem}

\begin{proof}
By Equation \eqref {E:swfrust}, $l(K_n,\s)=E^-(K_n,\s^\zeta)$ for a suitable switching function $\zeta$.  Switching means negating the signs of all edges in the cut $D$ between $\zeta\inv(-)$ and $\zeta\inv(+)$.  That adds $\delta=r(n-r) - 2|D\cap E^-(K_n,\s)|$ negative edges, where $r=|\zeta\inv(-)|$.  This number must be negative if $|E^-(K_n,\s^\zeta)|<|E^-(K_n,\s)|$, but then $0<r<n$ so $r(n-r)\geq n-1$/  Thus, $\delta \geq n-1 - 2|E^-(K_n,\s)| \geq 0$.
\end{proof}

The extrema for $k>3$ with larger frustration index are more difficult and are not known (to me, at least), with the obvious exception of the maxima for odd length.

\begin{prop}\label{P:maxknodd}
For odd $k$ with $3\leq k \leq n$, $\max_\s c_k^-(K_n,\s)=c_k^-(-K_n)=c_k(K_n)$.  If $(K_n,\s)$ is not antibalanced, then $c_k^-(K_n,\s)<c_k(K_n)$.
\end{prop}

\begin{proof}
It is clear that the maximum is attained by $-K_n$.  

The \emph{binary cycle space} of $K_n$ is the class of all subsets of $E$ that can be obtained from circles by the operation of symmetric difference.  It is generated by the circles of any one odd length (because those circles generate all quadrilaterals and the quadrilaterals permit shortening an odd circle to a triangle).  It follows that if $k$ is odd and $c_k^-(K_n,\s)=c_k(K_n)$, then all triangles are negative, so $(K_n,\s)$ is antibalanced.
\end{proof}

The next graphs to study could be the complete bipartite ones, also beginning with quadrilaterals.
\label{Cnum}

\item  Characterize the set ${\mathbf C}(\G) = \{c^-(\G,\s) : \s$ is a signature of  $\G\}$ of negative circle vectors of signatures of $\G$.
\\\A  Suppose a graph $\G$ of order $n$ has circles of lengths $0 < k_1 < k_2 < \cdots < k_m \leq n$ but not of any other lengths.  Then $\dim\mathbf C \leq m$ and we can think of the negative circle vectors as living in $\bbR^m$.  Here are some strengthenings of the question:
\label{Cvect}

\begin{enumerate}[\ref{Cvect}a.]
\item What is the dimension of $\mathbf C$?  In particular, when is $\dim\mathbf C=m$, the largest it could be?
\\\A  Schaefer and Zaslavsky \cite{Schaefer} find that we do have $\dim {\mathbf C}=m$ for $\G = K_n$ and $K_{r,s}$, where $m=n-2$ and $\min(r,s)-1$, respectively.  His method requires considerable symmetry of $\G$.
Since $\mathbf0 = c^-(\G,+) \in \mathbf C$, the linear and affine dimensions of $\dim\mathbf C$ are equal, which is a convenience.
\item What is the cone (with apex at the origin) generated by $\mathbf C$?  (That means finding the homogeneous inequalities satisfied by $\mathbf C$.)  What are the extreme rays of the cone?  Is there any combinatorial meaning to the extreme rays?
\\\A  Hard; unknown even for $K_n$.  All that is known is inequalities for individual components $c_k^-$ of the negative circle vectors.
\item What is the convex hull $\conv\mathbf C$?  (That means finding all inequalities satisfied by $\mathbf C$.)  What are the extreme points of $\conv\mathbf C$, and what is their significance?
\\\A  Harder than the cone!
\item What restrictions can one find for actual negative circle vectors?  For instance, Popescu found that a vector $c^-(K_n,\s) = (c_3^-,c_4^-,\ldots,c_n^-)$ cannot have an odd component except for $c_3^-$; he even found the smallest possible even part of each $c_k^-$ \cite{Pop2,PopTom1}.
\item Are there vectors other than $\mathbf0$ that can easily be guaranteed to be in $\mathbf C$?  
\\\A  The all-negative signature of a simple graph gives vector $c^-(-\G)=(c_3,0,c_5,0,\ldots)$ where $c_k$ is the total number of circles of length $k$.  Moreover, given a negative circle vector $c^-(\G,\s)$, $c^-(\G,-\s)$ is determined via $c_{2k}^-(\G,-\s)=c_{2k}^-(\G,\s)$ and $c_{2k-1}^-(\G,-\s)=c_{2k-1}(\G)-c_{2k-1}^-(\G,\s)$.  Schaefer and Zaslavsky made use of this fact in computing dimensions.
\end{enumerate}
\end{enumerate}

%%%%%%%%

%%%%%%%%%%%%%%%%%%%%%%%%%%%%
%%%%%%%%%%%%%%%%%%%%%%%%%%%%
\section{Eigenvalues}\label{eigen}

Even and odd circles have a surprising influence on eigenvalues of a graph.  Let $V=\{v_1,\ldots,v_n\}$.  The \emph{adjacency matrix} $A(\S)=(a_{ij})_{n\times n}$ has $a_{ij}=$ the number of positive edges $v_iv_j$ less the number of negative edges $v_iv_j$.  The \emph{Laplacian matrix} is defined as $L(\S)=D(|\S|)-A(\S)$, where $D(\G)$, the \emph{diagonal degree matrix} of a graph, is the diagonal matrix whose entry $d_{ii}$ is the degree of $v_i$.  Write $\mu_{\max}(\S)$ for the largest eigenvalue of $A(\S)$ and $\lambda_{\max}(\S)$ for the largest eigenvalue of $L(\S)$.  An unsigned graph can be treated as all positive, which gives its adjacency matrix $A(\G)=A(+\G)$ and Laplacian matrix $L(\G)=L(+\G)$, or as all negative, which gives the \emph{signless Laplacian matrix} $Q(\G)=L(-\G)$.  The last-named has attracted much attention since it was popularized by Cvetkovi\'c in \cite{Cvet} et al., but rarely in what I consider the proper perspective, which is that of signed graphs.

Why signed graphs?  Hou, Li, and Pan \cite{HLP} investigated the Laplacian matrices of signed graphs.  They discovered a remarkable fact.

\begin{thm}[\cite{HLP}]\label{T:domeigen}
For every signature of a connected graph $\G$ we have $\lambda_{\max}(-\G) \geq \lambda_{\max}(\G,\s)$, with equality if and only if $(\G,\s)$ is antibalanced.
\end{thm}

That $\lambda_{\max}(-\G)\geq\lambda_{\max}(+\G)$\ was known, but this theorem shows there is much more going on.  Reff \cite{CUG} proved the same result even more generally, for complex unit gains on $\G$.  Why is it so?  Reff (personal communication) observes that it is due to the fact that the nonzero off-diagonal entries of $L(-\G)$, which are $-1$, have the least real part possible for a complex number of modulus $1$.  I wonder if there is also a combinatorial explanation.

Signed graphs also seem likely to be implicated in eigenvalue phenomena discovered by Nikiforov and Yuan.  
Nikiforov \cite{Niki} found an eigenvalue property that implies a graph is not bipartite.  Assume $n\gg0$.  The theorem (simplified) says that if $\mu_{\max}(\G) > n^2/2$, then $\G$ cannot be bipartite because it contains a triangle.  In fact, it has a circle of every length $t\leq n/320$, in particular of every such odd length.  

Now let $\lambda_{\max}(-\G)$ be the largest eigenvalue of $Q(\G)=L(-\G)$.  Yuan and Nikiforov proved that if $\G$ contains no circle of a certain odd, or even, length $l$, then $\lambda_{\max}(-\G)$ has an explicit upper bound.    
Yuan \cite{Yuan} proved that if $k\geq3$, $n\geq 110k^2$, and $\lambda_{\max}(-\G) > \lambda_{\max}(-K_k \vee \overline K_{n-k})$, where $\vee$ denotes the join (i.e., the disjoint union together with all edges between the two graphs), then $G \supseteq C_{2k+1}$.  Nikiforov and Yuan \cite{NikiYu} proved a similar result for even circles $C_{2k}$.  
In other words, there are spectral criteria that imply existence of negative or positive circles, since the Laplacian is that of the all-negative signature, in which bipartiteness equals balance and the circles of interest are odd (that is, positive) or even (negative).

\begin{enumerate}[\Q 1. ]
\item\label{Cj:adj-unbal}
Does Nikiforov's theorem, with bipartiteness changed to balance, apply to all signed simple graphs that meet the conditions of the theorem of \cite{Niki}.  If not, to which ones does it apply?

\item\label{Cj:lap}
Do Yuan's theorem and that of Nikiforov and Yuan generalize to signed simple graphs that meet conditions similar to those of their theorems, with $C_{2k+1}$ replaced by a negative $C_l$ and $C_{2k}$ changed to a positive $C_l$?
\end{enumerate}

%%%%%%%%

%%%%%%%%%%%%%%%%%%%%%%%%%%%%
%%%%%%%%%%%%%%%%%%%%%%%%%%%%
\section{Signed Digraphs}\label{sd}

A signed digraph $(D,\s)$ is a directed graph $D$ with signed edges.  In a signed digraph we look at signed cycles, where by a \emph{cycle} I mean a connected subgraph with in-degree and out-degree $1$ at every vertex.  If every cycle is positive, we say $(D,\s)$ is \emph{cycle balanced}.  The properties of cycle balance are very different from those of balance; nevertheless one can ask the same questions.  

\subsection{Cycle balance vs.\ balance}\

There is a theorem that ties cycle balance tightly to balance.

\begin{thm}[Harary et al.\ \cite{HNC}]\label{T:strongbal}
A strongly connected signed digraph is balanced if it is cycle balanced.
\end{thm}

\begin{cor}\label{C:cyclebal}
A signed digraph is cycle balanced if and only if every strong component, ignoring directions, is balanced; i.e., has a Harary bipartition.
\end{cor}
\begin{proof}
Every cycle is contained in a strong component, so all strong components should be cycle balanced, but by the theorem, that means every strong component is balanced.
\end{proof}

\subsection{Cycles all negative (or positive)}\

Here is a pair of basic questions.

\begin{enumerate}[\Q $3^-$]
\item  In which signed digraphs are all cycles negative?
\\\A  The two articles I know of, \cite{HLM, Ch}, provide examples but the question remains open as far as I know.
\item  In which signed digraphs are they all positive?
\\\A  Corollary \ref{C:cyclebal} answers this.  The contrast between the positive and negative questions is striking.

\end{enumerate}

\subsection{Signed digraph frustration}\

Here are the questions about covering all negative, or positive, cycles by edges, or by vertices:

\begin{enumerate}[\Q $1^-$]
\item  What is the directed frustration index $l(D,\s)$, i.e., the smallest number of edges whose deletion results in cycle balance?
\\\A  If $S$ is a set of edges such that $(D,\s)\setm S$ is balanced, then $(D,\s)\setm S$ is also cycle balanced; therefore the directed frustration index is bounded above by the undirected frustration index $l(\G,\s)$, where $\G$ is the undirected underlying graph of $D$.  Under what conditions are they equal?
\label{sd-l}
\item  What is the directed frustration number $l_0(D,\s)$, i.e., the smallest number of vertices whose deletion results in cycle balance?
\\\A  As with the index, the directed frustration number is bounded above by $l_0(\G,\s)$.  When are they equal?
\label{sd-l0}
\item[\Q $\ref{sd-l}^+$]  What is the smallest number of edges that cover all positive cycles?  (Reminder:  Not necessarily all positive circles!)
\item[\Q $\ref{sd-l0}^+$]  What is the smallest number of vertices that cover all positive cycles?
\end{enumerate}

Montalva et al.\ \cite{MAG} showed that all four questions are NP-complete by reducing them to the known problems of covering all odd or even cycles in an unsigned digraph.  That still leaves a sufficiency of open questions.

It follows from Theorem \ref{T:strongbal} that for a strongly connected digraph, $l(D,\s)=l(\G,\s)$ and $l_0(D,\s)=l_0(\G,\s)$.  That partially, but only partially, answers Questions 1--2$^-$.

%%%%%%%%%%%%%%%%%%%%%%%%%%%%
%%%%%%%%%%%%%%%%%%%%%%%%%%%%

%%%%%%%%%%%%%%%%%%%%%%%%%%%%
%%%%%%%%%%%%%%%%%%%%%%%%%%%%
\section{The End}

And that concludes my survey.

\end{document}